\documentclass[12pt]{article} 

\usepackage{amsmath}
\usepackage{amsthm}
\usepackage{amsfonts}
\usepackage{mathrsfs}
\usepackage{stmaryrd}
\usepackage{setspace}
\usepackage{fullpage}
\usepackage{amssymb}
\usepackage{breqn}
\usepackage{enumitem}
\usepackage{bbold} 
\usepackage{authblk}
\usepackage{comment}
\usepackage{hyperref}
\usepackage{pgf,tikz}
\usepackage{graphicx}
\usetikzlibrary{decorations.pathreplacing,arrows}
\tikzstyle{vertex}=[circle,draw=black,fill=black,inner sep=0,minimum size=3pt,text=white,font=\footnotesize]

\newtheorem{theorem}{Theorem}

\newtheorem{definition}{Definition}
\newtheorem{claim}{Claim}
\newtheorem{conjecture}{Conjecture}

\newcommand{\C}{\mathcal{C}}

\newcommand{\h}{\mathcal{H}}

\newcommand{\abs}[1]{\left\lvert{#1}\right\rvert}

\DeclareMathOperator{\ex}{ex}
\newcommand{\subs}{\subseteq}

\title{On the Tur\'an number of some ordered even cycles}
\author[1]{Ervin Gy\H{o}ri \thanks{Research supported by NKFIH grant K116769.}}
\author[2]{D\'aniel Kor\'andi \thanks{Research supported by SNSF grants 200020-162884 and 200021-175977.}} 
\author[3]{Abhishek Methuku \protect\footnotemark[1]}
\author[2]{Istv\'an Tomon \protect\footnotemark[2]}
\author[1]{Casey Tompkins \protect\footnotemark[1]}
\author[1]{M\'at\'e Vizer \thanks{Research supported by NKFIH grant SNN 116095.}} 
\affil[1]{Alfr\'ed R\'enyi Institute of Mathematics, Hungarian Academy of Sciences. \newline
 \texttt{gyori.ervin@renyi.mta.hu, \{ctompkins496,vizermate\}@gmail.com}}
\affil[2]{\'{E}cole Polytechnique F\'{e}d\'{e}rale de Lausanne. 
 \texttt{ \{daniel.korandi,istvan.tomon\}@epfl.ch}}
\affil[3]{Central European University, Budapest.\par
 \texttt{abhishekmethuku@gmail.com}}

\begin{document}

\maketitle

\begin{abstract}
A classical result of Bondy and Simonovits in extremal graph theory states that if a graph on $n$ vertices contains no cycle of length $2k$ then it has at most $O(n^{1+1/k})$ edges. However, matching lower bounds are only known for $k=2,3,5$.

In this paper we study ordered variants of this problem and prove some tight estimates for a certain class of ordered cycles that we call \emph{bordered cycles}. In particular, we show that the maximum number of edges in an ordered graph avoiding bordered cycles of length at most $2k$ is $\Theta(n^{1+1/k})$.

Strengthening the result of Bondy and Simonovits in the case of 6-cycles, we also show that it is enough to forbid these bordered orderings of the 6-cycle to guarantee an upper bound of $O(n^{4/3})$ on the number of edges.

\end{abstract}


\noindent
{\bf AMS Subj.\ Class.\ (2010)}: 05D99, 51E99

\section{Introduction}

Tur\'an-type problems are generally formulated in the following way: one fixes some graph properties and tries to determine the maximum or minimum number of edges an $n$-vertex graph with the prescribed properties can have. These kinds of extremal problems have a rich history in combinatorics, going back to 1907, when
Mantel \cite{M1907} determined the maximum number of edges possible in a triangle-free graph. The systematic study of these problems began with Tur\'an \cite{T1941}, who generalized Mantel's result to arbitrary complete graphs.

For simple graphs $G$ and $H$, we say that $G$ is \emph{H-free} if $G$ does not contain $H$ as a subgraph. Given a graph $G$, its vertex set and edge set are denoted by $V(G)$ and $E(G)$, respectively. 

\begin{definition} Given a graph $H$ and a positive integer $n$, the \emph{Tur\'an number} of $H$ is
\begin{displaymath}
\ex(n,H):= \max \{|E(G)| \ : |V(G)|=n \textrm{ and G is H-free}\}.
\end{displaymath}
\end{definition}

Erd\H{o}s, Stone and Simonovits \cite{ES1965,ES1946} showed that the behavior of the Tur\'an  number of a general graph $H$ is determined by its chromatic number, $\chi(H)$, when $\chi(H) \ge 3$ . They proved that if $H$ is a simple graph, then 
\begin{displaymath}
\ex(n, H ) = \left(1-\frac{1}{\chi(H)-1}\right)\frac{n^2}{2} + o(n^2),
\end{displaymath}
which is an asymptotically correct result except when $H$ is bipartite.

In the bipartite case, one of the most natural problems is to estimate the Tur\'an number of even cycles. A classical result of Bondy and Simonovits \cite{BS1974} from 1974 is the following.
\begin{theorem}[Bondy--Simonovits] \label{BS}
For any $k \ge 2$, we have $\ex(n,C_{2k})=O(n^{1+\frac{1}{k}})$.\footnote{In this paper we use the standard asymptotic notations $O,o,\Theta$ understood as $n\to\infty$. The parameter $k$ is always assumed to be a constant.}
\end{theorem}

A major open question in extremal graph theory is whether this upper bound gives the correct order of magnitude. This was verified for $k=2,3$ and $5$. For example, the best known bounds for hexagons are due to F\"uredi, Naor and Verstra\"ete~\cite{FNV2006}, who proved that
\begin{equation}
\label{eq:c_6_free}
0.5338 \ n^{4/3} \le \ex(n, C_6) \le 0.6272 \ n^{4/3}.
\end{equation}
The corresponding girth problem was studied by Erd\H{o}s and Simonovits \cite{ES1982}, who conjectured that the same lower bound holds even if we also forbid cycles of shorter lengths.
\begin{conjecture}[Erd\H{o}s--Simonovits] \label{conj:es}
For any $k \ge 2$, we have 
\[ \ex(n,\{C_3,\dots,C_{2k}\})=\Theta(n^{1+\frac{1}{k}}). \]
\end{conjecture}

Once again, this conjecture is only known to hold for $k=2,3,5$. For a survey on the extremal number of bipartite graphs, we refer the reader to \cite{FS2013}.

\subsection{Ordered graphs and forbidden submatrices}

An \emph{ordered graph} is a simple graph $G=(V,E)$ with a linear ordering on its vertex set. We say that the ordered graph $H$ is an \emph{ordered subgraph} of $G$ if there is an embedding of $H$ in $G$ that respects the ordering of the vertices. The Tur\'an problem for a set of ordered graphs $\h$ asks the following. What is the maximum number $\ex_<(n,\h)$ of edges that an ordered graph on $n$ vertices can have without containing any $H\in\h$ as an ordered subgraph? When $\h$ contains a single ordered graph $H$, we simply write $\ex_<(n,H)$.

The systematic study of this problem was initiated by Pach and Tardos \cite{PT2006}. They noted that the following analog of the Erd\H{o}s--Stone--Simonovits result holds (see also \cite{BKV2003}):
\[
\ex_<(n,H) = \left(1-\frac{1}{\chi_<(H)-1}\right)\frac{n^2}{2} + o(n^2),
\]
where $\chi_<(H)$, the interval chromatic number of $H$, is the minimum number of intervals the (linearly ordered) vertex set of $H$ can be partitioned into, so that no two vertices belonging to the same interval are adjacent in $H$. This formula determines the asymptotics of the ordered Tur\'an number, except when $\chi_<(H)=2$.

The case $\chi_<(H)=2$ turns out to be closely related to a well-studied problem of forbidden patterns in 0-1 matrices. To formulate it, let $A_H$ be the bipartite adjacency matrix of $H$, where rows and columns correspond to the two intervals of $H$ (in the appropriate ordering), and 1-entries correspond to edges in $H$. Then we are interested in the maximum number of 1-entries in an $n\times m$ matrix that does not contain the pattern $A_H$ in the sense that $A_H$ is not a submatrix, nor can it be obtained from a submatrix by changing some 1-entries to 0-entries.

The problem of finding the extremal number of matrix patterns was introduced by F\"uredi and Hajnal \cite{FH1992} about 25 years ago, and several results have been obtained since then (see e.g. \cite{KTTW,MT04,PT2006,T05} and the references therein), although most of them concern matrices of \emph{acyclic} graphs. One notable exception is a result of Pach and Tardos \cite{PT2006} that establishes $\ex_<(n,\h)=\Theta(n^{4/3})$ for an infinite set of ordered cycles $\h$ that they call ``positive'' cycles. The definition of positive cycles is motivated by an incidence geometry problem, where they correspond to a class of forbidden configurations.

\medskip
In this paper we estimate $\ex_<(n,\h)$ for various finite sets $\h$ of ordered cycles that all come from the class of bordered cycles that we define as follows.
In an ordered graph with interval chromatic number 2 and intervals $U$ and $V$ ($U$ preceding $V$), we call the edge connecting the first vertex of $U$ and the last vertex of $V$ an \emph{outer border}, and the edge connecting the last vertex of $U$ and the first vertex of $V$ an \emph{inner border}. Then a \emph{bordered cycle} is an ordered cycle with interval chromatic number 2 that contains both an inner and an outer border. For example, out of the six ordered bipartite 6-cycles, three are bordered (see Figure~\ref{figure1}).

Let us emphasize that there is no containment relationship between our bordered cycles and the positive cycles of Pach and Tardos. For example, using the notation of Figure~\ref{figure1}, the positive 6-cycles are $C_6^1,C_6^3,C_6^I$ and $C_6^O$, whereas the bordered 6-cycles are $C_6^1,C_6^2$ and $C_6^3$.

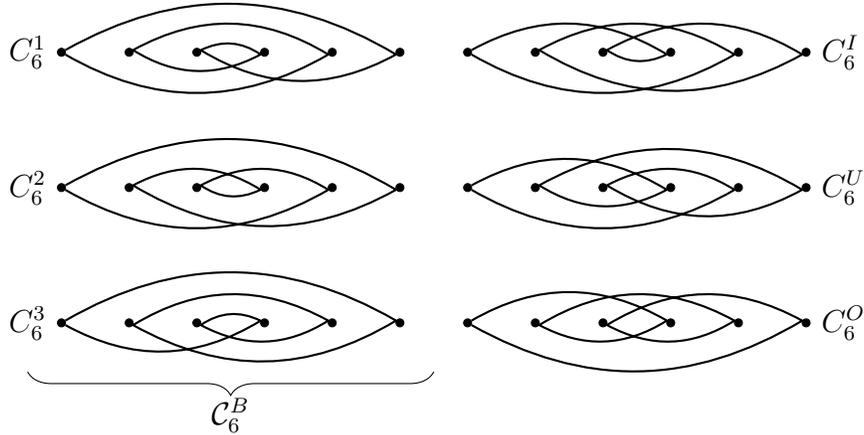
\begin{figure}[h]
\begin{center}
\begin{tikzpicture}[scale=.9]
\node[vertex,label=left:{$C_6^3$}] (A1) at (1,0) {};
\node[vertex] (B1) at (2,0) {};
\node[vertex] (C1) at (3,0) {};
\node[vertex] (D1) at (4,0) {};
\node[vertex] (E1) at (5,0) {};
\node[vertex] (F1) at (6,0) {};

\node[vertex,label=left:{$C_6^2$}] (A2) at (1,2) {};
\node[vertex] (B2) at (2,2) {};
\node[vertex] (C2) at (3,2) {};
\node[vertex] (D2) at (4,2) {};
\node[vertex] (E2) at (5,2) {};
\node[vertex] (F2) at (6,2) {};

\node[vertex,label=left:{$C_6^1$}] (A3) at (1,4) {};
\node[vertex] (B3) at (2,4) {};
\node[vertex] (C3) at (3,4) {};
\node[vertex] (D3) at (4,4) {};
\node[vertex] (E3) at (5,4) {};
\node[vertex] (F3) at (6,4) {};

\node[vertex] (A4) at (7,0) {};
\node[vertex] (B4) at (8,0) {};
\node[vertex] (C4) at (9,0) {};
\node[vertex] (D4) at (10,0) {};
\node[vertex] (E4) at (11,0) {};
\node[vertex,label=right:{$C_6^O$}] (F4) at (12,0) {};

\node[vertex] (A5) at (7,2) {};
\node[vertex] (B5) at (8,2) {};
\node[vertex] (C5) at (9,2) {};
\node[vertex] (D5) at (10,2) {};
\node[vertex] (E5) at (11,2) {};
\node[vertex,label=right:{$C_6^U$}] (F5) at (12,2) {};

\node[vertex] (A6) at (7,4) {};
\node[vertex] (B6) at (8,4) {};
\node[vertex] (C6) at (9,4) {};
\node[vertex] (D6) at (10,4) {};
\node[vertex] (E6) at (11,4) {};
\node[vertex,label=right:{$C_6^I$}] (F6) at (12,4) {};

\draw[thick] (A1) to[bend left] (F1) to[bend left] (B1) to[bend left] (E1) 
				 to[bend left] (C1) to[bend left] (D1) to[bend left] (A1);
\draw[thick] (A2) to[bend right] (E2) to[bend right] (C2) to[bend right] (D2)
				 to[bend right] (B2) to[bend right] (F2) to[bend right] (A2);
\draw[thick] (A3) to[bend right] (E3) to[bend right] (B3) to[bend right] (D3)
				 to[bend right] (C3) to[bend right] (F3) to[bend right] (A3);
\draw[thick] (A4) to[bend left] (D4) to[bend left] (B4) to[bend left] (E4)
				  to[bend left] (C4) to[bend left] (F4) to[bend left] (A4);
\draw[thick] (A5) to[bend right] (E5) to[bend right] (C5) to[bend right] (F5)
				 to[bend right] (B5) to[bend right] (D5) to[bend right] (A5);
\draw[thick] (A6) to[bend right] (E6) to[bend right] (B6) to[bend right] (F6)
				  to[bend right] (C6) to[bend right] (D6) to[bend right] (A6);

\draw[decoration={brace,mirror,amplitude=9pt,raise=-2pt},decorate] (.5,-.8) -- node[below=4pt] {$\C_6^B$} (6.5,-.8);

\end{tikzpicture}
	\caption{Bordered ($\C_6^B=\{C_6^1,C_6^2,C_6^3\}$), outbordered ($C_6^O$), unbordered ($C_6^U$) and inbordered ($C_6^I$) 6-cycles.
    Orderings shown in the same row can be obtained from each other by reversing the order of vertices in the second interval.
    }
     \label{figure1}
\end{center}
\end{figure}

\subsection{Our results}

Let $\C^B_{2k}$ be the set of bordered $2k$-cycles. Our main result determines the asymptotics of the maximum number of edges in an ordered graph with no bordered cycle of length up to $2k$. This can be thought of as an analog of Conjecture \ref{conj:es} for bordered cycles.

\begin{theorem} \label{thm:borderedgirth}
For every fixed integer $k>1$, 
\[  \ex_<(n,\{\C^B_4,\C^B_6,\dots,\C^B_{2k}\})=\Theta(n^{1+1/k}).  \]
\end{theorem}
\noindent

We do not know whether the bordered version of Theorem~\ref{BS} is true in general, i.e., if forbidding $\C^{B}_{2k}$ alone suffices to get the same asymptotic upper bound for the extremal number. However, we can prove this for $k=3$.

\begin{theorem} \label{thm:hexagon}
For bordered 6-cycles,
\[  \ex_<(n,\C^B_6)=\Theta(n^{4/3}).  \]
\end{theorem}

Actually, Theorem \ref{thm:hexagon} is an immediate consequence of Theorem \ref{thm:borderedgirth} and the fact that when $l-1$ divides $k-1$, then any $\C^{B}_{2k}$-free ordered graph contains a large $\C^{B}_{2l}$-free subgraph. This is established by the following theorem.

\begin{theorem}
\label{thm:c2lfree}
Let $k, l \ge 2$ be integers such that $k-1$ is divisible by $l-1$. Then any $\C^{B}_{2k}$-free ordered graph $G$ contains a $\C^{B}_{2l}$-free subgraph with at least $\frac{l-1}{k-1}$ fraction of the edges of $G$.
\end{theorem}

Note that for $l = 2$, Theorem \ref{thm:c2lfree} is a generalization of a theorem of K\"{u}hn and Osthus \cite{kuhn_osthus} which states that any bipartite $C_{2k}$-free graph $G$ contains a $C_4$-free subgraph with at least $1/(k-1)$ fraction of the edges of $G$. Indeed, if we order the vertices of $G$ so that all of the vertices in one of its color classes appear before the vertices of the other color class, then any $C_4$ in $G$ is bordered. Then Theorem \ref{thm:c2lfree} gives a $C_4$-free subgraph of $G$ that has at least $1/(k-1)$ fraction of the edges of $G$.

\medskip
This paper is organized as follows. In Section~\ref{sec:lowerbound} we prove the lower bound of Theorem~\ref{thm:borderedgirth} by constructing a dense ordered graph without short ordered cycles. The matching upper bound is shown in Section~\ref{sec:upperbound}. In Section~\ref{sec:c2lfree} we give a short proof of Theorem~\ref{thm:c2lfree}. We conclude the paper with some remarks and open problems in Section~\ref{sec:remarks}.

\section{Lower bound construction}
\label{sec:lowerbound}

Our construction is based on generalized Sidon sets defined as follows.

\begin{definition}
Let $k \ge 2$ be an integer. A set of integers $S$ is called a $B_k$-set if all $k$-sums of elements in $S$ are different, i.e., if for every integer $C$, there is at most one solution to
\[ x_1+x_2+\dots+x_k = C \]
in $S$, up to permuting the elements $x_i$ (the $x_i$ need not be distinct).\newline
We denote the maximum size of a $B_k$-set $S\subs \{1,2,\dots,n\}$ by $F_k(n)$.
\end{definition}
Note that this definition implies that if $x_1+x_2+\dots+x_l = x'_1+x'_2+\dots+x'_l$ for some $l \le k$ and $x_i, x'_i \in S$, then $\{x_1, x_2, \dots, x_l \} = \{x'_1, x'_2, \dots, x'_l\}$  as multisets.

Bose and Chowla \cite{BC1962} proved that 
\[ F_k(n) \ge n^{1/k} + o(n^{1/k}). \]

For a fixed $n \ge 1$, let $S \subset \{1,2,...,n\}$ be a $B_k$-set of size $\abs{S} = F_k(n)$.
Our construction will be an ordered graph $G$ on $4n$ vertices that we define through its bipartite adjacency matrix $A_G \in \{0,1\}^{2n\times 2n}$ as follows. For $1 \le i, j \le 2n$ we put 
\[
A_G(i,j) =\left\{
	\begin{array}{ll}
		1  & \mbox{if } i-j+n \in S \text{ and $1 \le i \le n$}\\
		0 &  \text{ otherwise.}
	\end{array}
\right.
\]

We will prove that $G$ contains no $2l$-cycle with a border edge for any $l\le k$. Note that the edges of a $2l$-cycle correspond to 1-entries in $S$ at coordinates $(i_1,j_1),\dots,(i_{2l},j_{2l})$, where
\begin{enumerate}
\item $i_{2s-1} = i_{2s}$ and $j_{2s-1} \ne j_{2s}$ for $s=1,2,\dots,l,$ and
\item $j_{2s} = j_{2s+1}$ and $i_{2s} \ne i_{2s+1}$ for $s=1,2,\dots,l$ (taking indices modulo $2l$).
\end{enumerate}
This readily implies  $\sum_{s=1}^{l} i_{2s}=\sum_{s=1}^{l} i_{2s-1}$ and $\sum_{s=1}^{l} j_{2s}=\sum_{s=1}^{l} j_{2s-1}$, and in particular
\begin{equation}
\label{eq:lhs_rhs}
\sum_{s=1}^{l} (i_{2s}-j_{2s}+n)= \sum_{s=1}^{l} (i_{2s-1}-j_{2s-1}+n).
\end{equation}

By the definition of $A_G$, we know that $i_s-j_s+n \in S$ for every $1 \le s \le 2l$, so both sides of \eqref{eq:lhs_rhs} are $l$-sums in $S$. But $S$ was a $B_k$-set with $l\le k$, so by the observation above, the two sums must have the same terms (possibly in a different order).

However, an outer (resp. inner) border of this cycle uniquely minimizes (resp. maximizes) $i_s-j_s$ over all $s=1,2,\dots,2l$ so if our cycle has a border edge, then there is a unique number among the terms, a contradiction.

Therefore, $G$ does not contain any cycle of length at most $2k$ with a border edge, and it is easy to see that it contains $nF_k(n)\ge n^{1 + 1/k} + o(n^{1+1/k})$ edges. This proves the lower bound of Theorem~\ref{thm:borderedgirth}.

\section{Upper bound}
\label{sec:upperbound}

Let $G=(V,E)$ be an ordered graph on the vertex set $V=\{x_{1}<x_{2}<\dots<x_{n}\}$ that avoids all cycles in $\C^{B}_{4},\dots,\C^{B}_{2k}$. We want to show that the number of edges $m$ in $G$ is $O(n^{1+1/k})$. Let us call a path $P=v_{0}v_{1}\dots v_{k}$ \emph{$k$-zigzag}, if $v_k<v_{k-2}<\dots<v_0<v_1<v_3<\dots<v_{k-1}$ for $k$ even, and $v_{k-1}<v_{k-3}<\dots<v_0<v_1<v_3<\dots<v_k$ for $k$ odd.
	
\begin{claim}\label{uniquepath}
The graph $G$ contains at most one $k$-zigzag path between any pair of vertices.
\end{claim}
\begin{proof}
Suppose to the contrary that $v_0\dots v_k$ and $v'_0\dots v'_k$ are two different $k$-zigzag paths such that $v_0=v'_0$ and $v_k=v_k'$. Let $s\in \{0,\dots,k\}$ be the largest index such that $v_i=v_i'$ for every  $i\in \{0,\dots,s\}$, and let $t$ be the smallest index larger than $s$ such that $v_t=v_t'$. Then $v_sv_{s+1}\dots v_tv_{t-1}'v_{t-2}'\dots v_s'$ are the consecutive vertices of a cycle of length $2(t-s)$, where $2\leq t-s\leq k$. Also, this cycle has two border edges, namely $v_{s}\min\{v_{s+1},v_{s+1}'\}$ and $\max\{v_{t-1},v_{t-1}'\}v_{t}$. But then $G$ contains a cycle in $\C^{B}_{2(t-s)}$, a contradiction.
\end{proof}
 	
This tells us that the number of $k$-zigzag paths in $G$ is at most $n^2$. Now let us bound the number of $k$-zigzag paths from below.
	
\begin{claim}\label{pathnumber}
The graph $G$ contains at least \[\frac{m^k}{k^k(3n)^{k-1}}\]  $k$-zigzag paths.	
\end{claim}
\begin{proof}
We will define a sequence of graphs $G_k,\dots,G_1\subs G$ recursively as follows. We set $G_k=G$, and we will obtain $G_{i-1}$ from $G_i$ by deleting the edges between each vertex and its $u=\lfloor m/2kn\rfloor$ largest and $u$ smallest neighbors.

More precisely, we define the left and right neighborhood of a vertex $x_s\in V$ in $G_i$ as
\[
L_i(x_s)=\{x_j:j<s,x_jx_s\in G_i\} \qquad \text{and} \qquad R_i(x_s)=\{x_j:j>s,x_sx_j\in G_i\},
\]
respectively. Also, let $L^+_i(x_s)$ be the $u$ smallest elements of $L_i(x_s)$, and let $R^+_i(x_s)$ be the $u$ largest elements of $R_i(x_s)$. (If $|L_i(x_s)|<u$, then we define $L^+_i(x_s)=L_i(x_s)$, and we do the same for $R^+_i(x_s)$.) We then set 
\[
G_{i-1} = G_i \setminus \left( \bigcup_{s=1}^{n}\{x_jx_s:x_j\in L^{+}_{i}(x_s)\} \cup \bigcup_{s=1}^{n}\{x_sx_j:x_j\in R^{+}_{i}(x_s)\} \right).
\]

Let us collect some properties of the graphs $G_i$.
\begin{enumerate}
 \item We delete at most $2nu\le m/k$ edges from $G_i$ to obtain $G_{i-1}$, so we have $|E(G_i)|\ge mi/k$ for every $i$, and in particular, $|E(G_1)|\ge m/k$.          
 \item For every $x\in V$ and every $i$, we have $L_{2i}^{+}(x)<L_{2i+2}^{+}(x)$ and $R_{2i+1}^{+}(x)<R_{2i-1}^{+}(x)$, where we write $A<B$ for some sets $A,B\subset V$ if $\max A<\min B$.
 \item For every $x\in V$, if $L_{2i-1}(x)$ is non-empty, then $|L_{2i}^{+}(x)|=u$. Similarly, if $R_{2i}(x)$ is non-empty, then $|R_{2i+1}^{+}(x)|=u$.
\end{enumerate}
     
Now we show that for every edge $f=v_0v_1\in G_1$, there are at least $u^{k-1}$ $k$-zigzag paths starting with $f$. Observe that every sequence of vertices $v_0,v_1,\dots,v_k$ satisfying $v_i\in L_i^+(v_{i-1})$ for $i$ even, and $v_i\in R_i^+(v_{i-1})$ for $i$ odd is a $k$-zigzag path by property 2. Also, the number of such paths is exactly $u^{k-1}$ by property 3. 
Hence, using $u\ge m/3kn$, we have at least $|E(G_1)|u^{k-1}>m^k/k^k(3n)^{k-1}$ different $k$-zigzag paths in $G$.
\end{proof}
 
Now comparing our lower and upper bound for the number of $k$-zigzag paths in $G$, we arrive at the inequality 
\[ n^2 \ge \frac{m^k}{k^k(3n)^{k-1}}, \]
which yields $m<3kn^{1+1/k}$. \qed

\section{Deleting small cycles}
\label{sec:c2lfree}

Our proof of Theorem \ref{thm:c2lfree} is inspired by the proof of Gr\'osz, Methuku and Tompkins in \cite{GMT} on deleting $4$-cycles, which is a simple proof of a theorem of K\"{u}hn and Osthus \cite{kuhn_osthus}. We make use of the following result of Gallai \cite{G1968} and Roy \cite{R1967}.

\begin{theorem}[Gallai--Roy] \label{thm:gallai}
If a directed graph $G$ contains no directed path of length $h$ then $\chi(G)\le h$.
\end{theorem}

\begin{proof}[Proof of Theorem \ref{thm:c2lfree}]
Let $G=(V,E)$ be an ordered graph which is $\C_{2k}^{B}$-free, where the elements of $V$ are $x_{1}<\dots<x_{n}$. Define the directed graph $H$ on $E$ as a vertex set such that for $f,f'\in E$, $\overrightarrow{ff'}$ is a directed edge of $H$ if there exists a bordered $2l$-cycle with outer border $f$ and inner border $f'$. Note that $H$ is acyclic, because if $\overrightarrow{ff'}\in E(H)$, where $f=ab$, $f'=a'b'$, $a<b$ and $a'<b'$, then $a<a'<b'<b$. 

We show that the longest directed path in $H$ has length less than $h=\frac{k-1}{l-1}$. Suppose to the contrary that there is a directed path  $f_1\dots f_{h+1}$ in $H$. Then for every $i=1,\dots, h$, there is a bordered cycle $C_{i}$ with outer border $f_{i}$ and inner border $f_{i+1}$. Then it is easy to see that $\left(\bigcup_{i=1}^{h}C_{i}\right)\setminus \{f_{2},\dots,f_{h}\}$
is a bordered cycle of length $2lh-2h+2=2k$, with outer border $f_{1}$, and inner border $f_{h+1}$, contradicting the choice of $G$.

Hence we can apply Theorem~\ref{thm:gallai} to get a proper $h$-coloring of $H$. Here the largest color class $E_0\subs E$ is an independent set of size at least $\frac{l-1}{k-1}\abs{E}$, so there is no cycle in $\C^{B}_{2l}$ that has all its edges in $E_0$. The edges of $E_0$ will then form an ordered subgraph of $G$ that satisfies our conditions.
\end{proof}

\section{Concluding remarks}
\label{sec:remarks}

Note that Theorem \ref{thm:hexagon} is stronger than the $k=3$ case of Theorem~\ref{BS} because it only forbids three out of the six orderings of the hexagon. In fact, it is enough to forbid two orderings of the hexagon to achieve the same asymptotic bound.

\begin{theorem}
\label{thm:2hexagons}
Let $\mathscr{C}_{1}=\{C^{2}_{6},C^{1}_{6}\}$, $\mathscr{C}_{2}=\{C^{2}_{6},C^{3}_{6}\}$, $\mathscr{C}_{3}=\{C^{U}_{6},C^{I}_{6}\}$, and $\mathscr{C}_{4}=\{C^{U}_{6},C^{O}_{6}\}$. For any $i \in \{1,2,3,4\}$, we have $\ex_{<}(n,\mathscr{C}_{i})=\Theta(n^{4/3}).$
\end{theorem}

\begin{proof}[Sketch of the proof.] It is enough to show that every $\mathscr{C}_i$-free ordered graph $G$ on $2n$ vertices has $O(n^{4/3})$ edges between the first $n$ and the last $n$ vertices. Indeed, an inductive argument applied to the two halves of $G$ then yields a $O(n^{4/3})$ upper bound on the total number of edges, as well. We first show this for $\mathscr{C}_1$, so let $G$ be an ordered graph on the vertex set $A\cup B$ with $|A|=|B|=n$ and $A<B$ that has no edges induced by $A$ or $B$, and avoids $C^1_6$ and $C^2_6$.

Note that $G$ cannot contain two bordered 4-cycles such that the inner border of one is the outer border of the other, because they would create a copy of $C^2_6$. So by the argument of Theorem \ref{thm:c2lfree}, we can assume that $G$ does not contain any bordered $4$-cycle. The rest of the proof follows that of Theorem \ref{thm:borderedgirth}; we only need that, analogously to Claim \ref{uniquepath}, if for some $x\in A,y\in B$, we have two 3-zigzag paths $P_1$ and $P_2$ from $x$ to $y$, then $P_1\cup P_2$ is either $C^1_6$ or $C^2_6$, or it induces a bordered 4-cycle. So we once again get that the number of $3$-zigzag paths in $G$ is at most $n^2$, and can finish the argument as before.

To obtain an upper bound on $\ex_{<}(n,\mathscr{C}_{i})$ for $i \in \{2, 3, 4\}$, note that we can obtain each $\mathscr{C}_{i}$ from $\mathscr{C}_1$ by reversing the order of the vertices in one (or both) of the color intervals. This means, for example, that the graph $G$ above is $\mathscr{C}_1$-free if and only if the graph $G'$, obtained from $G$ by reversing the order of vertices in $B$, is $\mathscr{C}_3$-free. In particular, such a $G'$ has $O(n^{4/3})$ edges, and a similar reduction works for all other $i$.
\end{proof}

As we mentioned before, Pach and Tardos \cite{PT2006} showed that $\ex_{<}(n,\mathscr{C})=\Theta(n^{4/3})$ for a certain set of cycles that they call ``positive''. They also asked if it would be enough to forbid the positive $6$-cycles (i.e., $C^{1}_{6},C^{3}_{6},C^{O}_{6},C^{I}_{6}$) to get the same upper bound. More generally, we propose the following conjecture.

\begin{conjecture} \label{conj:hex}
Let $C$ be an ordered 6-cycle of interval chromatic number $2$. Then
\[ \ex_{<}(n,C)=\Theta(n^{4/3}). \]
\end{conjecture}

Finally, let us remark that while we are unable to prove that $\ex_{<}(n,\C^{B}_{2k})=O(n^{1+1/k})$, there is certainly no absolute constant $\varepsilon>0$ such that $\ex_{<}(n,\C^{B}_{2k})\ge n^{1+\varepsilon}$ for every $k$: 
\begin{theorem}
 There exists a sequence of positive real numbers $(\lambda_k)_{k=2,3,\dots}$ such that $\ex_{<}(n, \C^{B}_{2k})=O(n^{1+\lambda_k})$ and $\liminf_{k\rightarrow \infty}\lambda_k=0$.	
\end{theorem}

\begin{proof}
We will show that we can choose $\lambda_{(m-1)!+1}=1/m$. Let $k=(m-1)!+1$ and let $G$ be a $\C^{B}_{2k}$-free ordered graph with $n$ vertices. Then for any $2 \le l \le m$, $l-1$ divides $k-1$. Therefore, applying Theorem \ref{thm:c2lfree} repeatedly, we obtain a subgraph $G'$ of $G$ such that $G'$ is $\{\C^{B}_{4},\C^{B}_{6},\dots,\C^{B}_{2m}\}$-free, and $G'$ has at least $(m-1)!/(k-1)^{m-1}$ proportion of the edges of $G$. But then, by Theorem \ref{thm:borderedgirth}, $|E(G')|=O(n^{1+1/m})$, and thus $|E(G)|=O(n^{1+1/m})$, as well.
\end{proof}

\medskip
After our paper was submitted, we learned that Timmons \cite{T12} also studied the Tur\'an number of ordered cycles, and observed that $\ex_<(n,\C_{2k})=O(n^{1+1/k})$ for the family $\C_{2k}$ of all ordered $2k$-cycles with interval chromatic number 2. On the other hand, he found the construction we presented in Section~\ref{sec:lowerbound}, and asked whether a matching upper bound holds, i.e., if the upper bound $O(n^{1+1/k})$ holds when only the ordered $2k$-cycles with an inner or an outer border are forbidden. Our Theorem~\ref{thm:hexagon} (or Theorem~\ref{thm:2hexagons}) answers this question positively for $k=3$ (for an even smaller subfamily), and Theorem~\ref{thm:borderedgirth} answers a variant for every $k$ where shorter cycles are also forbidden. We kept the construction in this paper for completeness.

\medskip
\noindent
\textbf{Acknowledgments.} \\
We thank G\'abor Tardos for fruitful discussions, and Craig Timmons for bringing \cite{T12} to our attention. We also thank the referees for their valuable comments.

\end{document}